\documentclass[a4paper,11pt,leqno]{article}
\usepackage{amsmath,amssymb}
\usepackage{amsfonts}
\usepackage{eucal}
\usepackage{amsthm}
\numberwithin{equation}{section}

\newcommand{\bigpare}[1]{\bigl(#1\bigr)}
\newcommand{\biggpare}[1]{\biggl(#1\biggr)}
\newcommand{\Bigpare}[1]{\Bigl(#1\Bigr)}

\newcommand{\bigbra}[1]{\bigl\{#1\bigr\}}

\newcommand{\Bigbrac}[1]{\Bigl[#1\Bigr]}

\newcommand{\bigset}[2]{\bigl\{#1\bigm|#2\bigr\}}

\newcommand{\norm}[1]{\| #1 \|}
\newcommand{\bignorm}[1]{\bigl\| #1 \bigr\|}
\newcommand{\Bignorm}[1]{\Bigl\| #1 \Bigr\|}
\newcommand{\biggnorm}[1]{\biggl\| #1 \biggr\|}

\newcommand{\bigabs}[1]{\bigl| #1 \bigr|}
\newcommand{\Bigabs}[1]{\Bigl| #1 \Bigr|}
\newcommand{\biggabs}[1]{\biggl| #1 \biggr|}

\newcommand{\jap}[1]{\langle #1 \rangle}

\def\a{\alpha}
\def\b{\beta}
\def\c{\gamma}
\def\d{\delta}

\def\g{\psi}

\def\l{\lambda}
\def\m{\mu}
\def\n{\nu}

\def\s{\sigma}

\def\x{\xi}
\def\y{\eta}

\newcommand{\F}{\Phi}
\newcommand{\G}{\Psi}

\renewcommand{\L}{\Lambda}

\renewcommand{\S}{\Sigma}

\def\re{\mathbb{R}}

\def\ze{\mathbb{Z}}

\def\pa{\partial}

\newcommand{\supp}{\mathrm{{supp}}}

\newcommand{\WF}{\mathrm{WF}}

\newtheorem{thm}{Theorem}
\newtheorem{lem}[thm]{Lemma}
\newtheorem{prop}[thm]{Proposition}
\newtheorem{cor}[thm]{Corollary}

\theoremstyle{definition}

\newtheorem{ass}{Assumption}

\theoremstyle{remark}
\newtheorem{rem}{Remark}

\numberwithin{equation}{section}

\title{Propagation of singularities for Schr\"odinger equations with 
modestly long range type potentials\footnote{2010 Mathematics Subject Classification 35B65, 35P25, 35A27.}}
\author{
Kazuki H{\sc orie}\footnote{Graduate School of Mathematical Sciences, 
University of Tokyo, 3-8-1 Komaba, Meguro Tokyo, 
153-8914 Japan.} 
\ and Shu N{\sc akamura}%
\footnote{Graduate School of Mathematical Sciences, 
University of Tokyo, 3-8-1 Komaba, Meguro Tokyo, 
153-8914 Japan. 
E-mail: {\tt shu@ms.u-tokyo.ac.jp}.  
Partially supported by JSPS Grant Kiban (A) 21244008.} }

\begin{document}
\maketitle

\begin{abstract}
In a previous paper by the second author \cite{Na2}, we discussed a characterization of the microlocal singularities for 
solutions to Schr\"odinger equations with long range type perturbations, using solutions to a 
Hamilton-Jacobi equation. In this paper we show that we may use Dollard type approximate solutions to the 
Hamilton-Jacobi equation if the perturbation satisfies somewhat stronger conditions. 
As applications, we describe the propagation of microlocal singularities for $e^{itH_0}e^{-itH}$ when the 
potential is asymptotically homogeneous as $|x|\to\infty$, where $H$ is our Schr\"odinger operator, 
and $H_0$ is the free Schr\"odinger operator, i.e., $H_0=-\frac12 \triangle$. We show 
$e^{itH_0}e^{-itH}$ shifts the wave front set if the potential $V$ is asymptotically homogeneous of order 1, 
whereas $e^{itH}e^{-itH_0}$ is smoothing if $V$ is asymptotically homogenous of order $\b\in (1,3/2)$. 
\end{abstract}


\section{Introduction}

We consider Schr\"odinger operator with variable coefficients on $\re^d$, $d\geq 1$: 
\[
H=-\frac12 \sum_{m,n=1}^d \frac{\pa}{\pa x_m} a_{mn}(x) \frac{\pa}{\pa x_n} +V(x)
\quad\text{on }L^2(\re^d).
\]
We assume the coefficients satisfy long range type conditions in the following sense: 

\begin{ass}
$a_{mn}$, $V\in C^\infty(\re^d;\re)$, and let $\m>0$. For any multi-index $\a\in \ze_+^d$, 
there is $C_\a>0$ such that 
\[
\bigabs{\pa_x^\a (a_{mn}(x)-\d_{mn})}\leq C_\a \jap{x}^{-\m-|\a|}, 
\quad\bigabs{\pa_x^\a V(x)}\leq C_\a \jap{x}^{2-\m-|\a|}, 
\]
for $x\in\re^d$, where $\pa_x=\pa/\pa x$ and $\jap{x}=(1+|x|^2)^{1/2}$. 
We also assume $(a_{mn}(x))_{m,n=1}^d$ is a positive symmetric matrix for 
each $x\in\re^d$. 
\end{ass}

Then it is well-known that $H$ is a self-adjoint operator with the domain $H^2(\re^d)$, 
the Sobolev space of order 2. We consider solutions to the Schr\"odinger eqaution:
\[
i\frac{\pa}{\pa t} \g(t)=H\g(t), \quad \g(0)=\g_0\in L^2(\re^d). 
\]
By the Stone theorem, the solution is given by $\g(t)=e^{-itH}\g_0\in  L^2(\re^d)$. 
It is also well-known that the singularities of the solution propagate with infinite speed, 
and hence the \textit{propagation of singularities theorem}\/ analogous to the 
solutions to the wave equation cannot hold. In \cite{Na2}, it is proved that the 
wave front set of $e^{-itH}\g_0$ is described in terms of the wave front set of 
$e^{-i\F(t,D_x)}\g_0$, where $\F(t,\x)$ is a solution to the Hamilton-Jacobi equation:
\[
\frac{\pa \F}{\pa t}(t,\x)= p\biggpare{\frac{\pa\F}{\pa\x}(t,\x),\x}, \quad t\in\re,\x\in\re^d, 
\]
and 
\[
p(x,\x) =\frac12\sum_{m,n=1}^d a_{mn}(x)\x_m\x_n +V(x), \quad x,\x\in\re^d, 
\]
 is the symbol of the Schr\"odinger operator $H$. 
 
 The purpose of this paper is to show that if $\m>1/2$, we may employ the Dollard type approximate solution, 
 or a modifier:
 \[
 \F_D(t,\x)= \int_0^t p(s\x,\x) ds
 \]
 to characterize the microlocal singularities of the solution. 
 One advantage to use the Dollard type modifier is that it is easy to compute, and hence 
 we can describe the behavior of the propagation explicitly for several cases. In particular, if $V(x)$ 
 is asymptotically homogeneous of order $\b\in [1,3/2)$, we give explicit characterization of the 
 wave front set of $e^{-itH}\g_0$ in terms of $e^{-itH_0}\g_0$. 
 
 Now we state our main result. Let $\exp(tH_k)$ be the Hamilton flow generated by 
 $k(x,\x)=\frac12 \sum_{m,n=1}^d a_{mn}(x)\x_m\x_n$, i.e., 
 \[
 (x(t),\x(t)) =\exp(tH_k)(x_0,\x_0)
 \]
 if $(x(t),\x(t))$ is the solution to the Hamilton equation:
 \[
 x'(t)=\frac{\pa k}{\pa \x}(x(t),\x(t)), \quad \x'(t)=-\frac{\pa k}{\pa x}(x(t),\x(t)), 
 \quad x(0)=x_0,\ \x(0)=\x_0.
 \]
 Under Assumption~A, it is well-known that $\exp(tH_k)$ is a diffeomorphism in $\re^d$ for any $t\in\re$. 
 We assume all the trajectories are nontrapping in the following sense: 

\begin{ass}
Let $ (x(t),\x(t)) =\exp(tH_k)(x_0,\x_0)$ with $\x_0\neq 0$. Then $|x(t)|\to\infty$ as $t\to\pm\infty$. 
\end{ass}

\begin{rem}
We may assume this nontrapping condition only for $(x_0,\x_0)$ we are looking at, but we assume 
global nontrapping condition to simplify the notation. 
\end{rem}

As we have mentioned above, we suppose Assumption~A with $\m>1/2$. For the later applications, 
it is convenient to suppose $V$ is decomposed to a \textit{long range part}\/ and 
a \textit{short range part}.

\begin{ass}
$a_{mn}(x)$ and $V(x)$ satisfies Assumption~A with $\m>1/2$. Moreover, 
$V(x)=V^{(L)}(x)+V^{(S)}(x)$, where $V^{(L)}$ satisfies Assumption~A with $\m>1/2$, and 
$V^{(S)}$ satisfies
\[
\bigabs{\pa_x^\a V^{(S)}(x)}\leq C_\a\jap{x}^{2-\n-|\a|}, \quad x\in\re^d, 
\]
with $\n>1$, $C_\a>0$. 
\end{ass}

Under these conditions, we can show the existence of the classical (long range) scattering. We set
\begin{align}
&p^{(L)}(x,\x)= k(x,\x)+V^{(L)}(x), \label{eq-Hamiltonians}\\
&\F(t,\x) =\int_0^t p^{(L)}(s\x,\x)ds, \quad \G(t,\x)=\int_0^t k(s\x,\x) ds. \label{eq-Dollard}
\end{align}

\begin{prop}
Let $(x_0,\x_0)\in\re^d\times(\re^d\setminus\{0\})$, and set $(x(t),\x(t))=\exp(tH_k)(x_0,\x_0)$. 
Then 
\[
x_\pm =\lim_{t\to\pm\infty} \bigpare{x(t)-\nabla_\x \G(t,\x)}, 
\quad \x_\pm =\lim_{t\to\pm\infty} \x(t)
\]
exist. If we write 
\[
W^{cl}_\pm\ :\ (x_\pm,\x_\pm)\mapsto (x_0,\x_0), 
\]
then $W^{cl}_\pm$ are diffeomorphisms in $\re^d\times (\re^d\setminus\{0\})$. 
\end{prop}

We prove Proposition~1 in Section~2. 

The next theorem is our main result: We denote the wave front set of $u\in\mathcal{S}'(\re^d)$ by 
$\WF(u)$. We write $D_x=-i\pa/\pa x$, and $F(D_x)$ denotes the Fourier multiplier with the symbol 
$F(\x)$. 

\begin{thm}
Suppose Assumptions~ B and C. Then for any $u\in L^2(\re^d)$, 
\[
\WF\bigpare{e^{i\F(t,D_x)} e^{-itH}u} = \bigpare{W_\pm^{cl}}^{-1}(\WF(u))
\]
for $\pm t>0$. 
\end{thm}

By replacing $u$ by $e^{itH}u$, we obtain the following corollary: 

\begin{cor}
Under the same assumptions as Theorem~2, 
\[
\WF\bigpare{e^{-itH}u} = W_\mp^{cl}\bigpare{\WF\bigpare{e^{-i\F(t,D_x)} u}}
\]
for $\pm t>0$. In other words, $(x_0,\x_0)\in\WF\bigpare{e^{-itH}u}$ if and only if
$(x_\mp,\x_\mp)\in \WF\bigpare{ e^{-i\F(t,D_x)}u}$ when $\pm t>0$. 
\end{cor}

In the remaining of this Introduction, we consider the case $V(x)$ is asymptotically homogeneous 
as $|x|\to\infty$. Here we suppose $a_{mn}(x)=\d_{mn}$ for the sake of simplicity, though it is 
not really necessary. In this case,  $\G(t,\x)=\frac{t}{2}|\x|^2$, Assumption~B is satisfied, 
and the classical wave map $W^{cl}_\pm$ is the identity map: $(x_\pm,\x_\pm)=(x,\x)$. 

We first suppose $V(x)$ is homogeneous of order 1, i.e., 
\begin{equation}\label{eq-Homogeneous}
V^{(L)}(x) =|x| V^{(L)}(\hat x) \quad \text{if}\quad |x|\geq 1, 
\end{equation}
where $\hat x={x}/{|x|}\in S^{d-1}$. This condition implies $\pa_x V^{(L)}(x)= \pa_x V^{(L)}(\hat x)$ 
if $|x|\geq 1$, i.e., $\pa_x V^{(L)}(x)$ depends only on the direction $\hat x$ of $x$. 
We set
\[
S^\pm_\s(x,\x)= (x\pm \s\pa_x V^{(L)}(\pm \hat\x),\x), 
\quad \s\in\re, x,\x\in\re^d,\x\neq 0.
\]
$S^\pm_\s$ is the Hamilton flow generated by $V^{(L)}(\pm\x)$ : 
$S^\pm_\s=\exp\bigpare{\pm\s H_{V^{(L)}(\pm \x)}}$, if $|\x|\geq 1$. 

\begin{thm}
We suppose Assumption C, $a_{ij}(x)=\d_{ij}$,  \eqref{eq-Homogeneous}, and let $u\in L^2(\re^d)$. Then 
\[
\WF \bigpare{e^{itH_0} e^{-itH}u} =S^\pm_{(-t^2/2)}(WF(u)), \quad \pm t>0, 
\]
and hence 
\[
\WF\bigpare{e^{-itH}u} = S^\mp_{t^2/2}\bigpare{\WF\bigpare{e^{-itH_0}u}}.
\]
\end{thm}

Theorem~4 implies the wave front set of the solution shifts according to 
the Hamilton flow generated by $V^{(L)}(\x)$, if the metric is flat and the potential is asymptotically 
homogeneous of order 1. 

We now turn to the case when $V(x)$ is asymptotically homogeneous of order $\b\in (1, 3/2)$. 
In this case, the behavior of the singularities are quite different. 
Since the quantization of $\exp(\s H_V(\x))$, $e^{-i\s V(D_x)}$ has diffusive properties similar to 
the free Schr\"odinger evolution group, we expect the vanishing of the singularities for 
$e^{itH_0}e^{-itH}u$ if $u$ decays rapidly as $|x|\to\infty$. In fact, we can prove the following: 

\begin{thm}
Suppose Assumptions B, C and 
\[
V^{(L)}(x) = |x|^\b V^{(L)}(\hat x) 
\quad \text{for}\quad |x|\geq 1,
\]
with $\b\in (1,3/2)$. We suppose, moreover, that $\nabla V^{(L)}(\hat x)\neq 0$ for $\hat x\in S^{d-1}$. 
If $e^{-itH_0}u\in L^{2,\infty}(\re^d) :=\bigset{f\in L^2(\re^d)}{\jap{x}^m f(x)\in L^2(\re^d) \text{ \rm for any }m}$, then  $e^{-itH}u\in C^\infty(\re^d)$, where $t\neq 0$.  
\end{thm}

Thus we observe that the propagation of singularities for $e^{-itH}$ depends drastically on the growth rate 
of the potential at infinity. 


\smallskip
The singularities of the solutions to Schr\"odinger equation has been studied by many mathematicians, 
mostly the smoothing properties with applications to nonlinear problems in mind. 
The explicit characterization of the wave front set of solutions was 
obtained relatively recently by Hassel, Wunsch \cite{HW}, Nakamura \cite{Na1}, \cite{Na2}, 
Ito, Nakamura \cite{IN1} and Martinez, Nakamura, Sordoni \cite{MNS}, \cite{MNS2} under different conditions. 
We note that closely related results had been obtained for perturbed harmonic oscillators (with constant 
principal part. See \=Okaji \cite{Ok}, Doi \cite{Do} and references therein), 
which does not require scattering theoretical framework (See also Mao, Nakamura \cite{MaoN} for 
non constant principal part cases). 

We call a Schr\"odinger operator satisfying Assumption~A with $\m>1$ {\it short range type}, 
though the potential $V(x)$ may be unbounded as $|x|\to\infty$, and not necessarily 
short range in the sense of scattering theory. The previous works above concern short range 
cases, except \cite{Na2}, \cite{MNS2}. We call a Schr\"odinger operator satisfying Assumption~A with $\m\in(0,1]$
{\it long range type}. In order to describe the microlocal singularities of solutions, we need to employ 
the framework of long range scattering theory (for the classical flow). In \cite{Na2}, a solution to 
the Hamilton-Jacobi equation at high energy is constructed for that purpose, but the construction is 
rather long and not easily computed. In this paper we use simpler Dollard type modifier 
(see, e.g., \cite{RS} Section XI.9) to characterize the 
microlocal singularities, and we wish this construction clarifies the analysis of \cite{Na2}. 

The result for the asymptotically homogeneous case, Theorem~4, is closely related to the work 
by Doi \cite{Do}, and Theorem~4 may be considered as a direct analogue of his result in our setting. 

We also remark that we can actually prove $e^{i\F(t,D_x)} e^{-itH}$ is a Fourier integral operator 
(in a slightly generalized sense) using the method of Ito, Nakamura \cite{IN3}, which we do not 
discuss in this paper. 


\smallskip
The paper is constructed as follows: 
We discuss the scattering theory for the classical mechanical flow in Section~2. We prove Theorem~2 in 
Section~3, mostly following the argument of \cite{Na2}. We prove properties of our main examples, 
i.e., Theorems~4 and 5 in Section~4. 

\smallskip
Throughout this paper, we use the following notation: We mainly work in $L^2(\re^d)$, and the norm 
$\norm{\cdot}$ denotes the $L^2$-norm unless otherwise specified. We write $\jap{x}=(1+|x|^2)^{1/2}$, 
which is standard in the microlocal analysis. $\ze_+=\{0,1,2,\dots\}$ denotes the set of non-negative integers, 
and $\ze_+^d$ is the set of multi-indices. For $\a=(\a_1,\dots,\a_d)\in\ze_+^d$, we denote $|\a|=\sum_j \a_j$.  
$S^m_{1,0}$ denotes the standard pseudodifferential operator symbol class, i.e., 
$a\in S^m_{1,0}$ means $a\in C^\infty(\re^d\times\re^d)$ such that for any 
$\a,\b\in\ze_+^d$, $K\Subset\re^d$, 
\[
\bigabs{\pa_x^\a\pa_\x^\b a(x,\x)}\leq C_{\a\b}\jap{\x}^{m-|\b|}, \quad x\in K, \x\in\re^d
\]
with some $C_{\a\b}>0$. $\mathcal{S}(\re^d)$ denotes the set of Schwartz functions. 
We denote various constants by $C$, which may change line to line. 

\smallskip
\noindent
{\bf Acknowledgement:} This paper is based on the master thesis of the first author (KH). 


\section{Classical mechanics and the high energy asymptotics}

Here we consider the existence of the scattering theory and the high energy asymptotics for the classical 
mechanical flow. We denote
\begin{align*}
&(x(t;x_0,\x_0),\x(t;x_0,\x_0))= \exp(tH_p)(x_0,\x_0), \\
& (y(t;x_0,\x_0)),\y(t;x_0,\x_0))=\exp(tH_k)(x_0,\x_0). 
\end{align*}
We first prove Proposition~1. We always suppose Assumptions~A, B with $1/2<\m\leq 1$. 

\begin{proof}[Proof of Proposition 1] 
We fix $(x_0,\x_0)\in\re^d\times \re^d$, $\x_0\neq 0$. It is well-known that
\begin{equation}\label{eq2-1}
|y(t;x_0,\x_0)|\geq c|t|-C, \quad t\in\re,
\end{equation}
with some $c,C>0$ if $(x_0,\x_0)$ is nontrapping (see, e.g., \cite{Na1} Lemma~2, \cite{Na2} Proposition~2.1). 
By the Hamilton equation, we have 
\begin{equation}\label{eq2-2}
\begin{aligned}
\biggabs{\frac{d}{dt} \y_j(t;x,x_0)} 
&= \biggabs{\frac12 \sum_{m,n=1}^d {\frac{\pa a_{mn}}{\pa x_j}}(y(t))\,\y_m(t)\y_n(t)} \\
&\leq C\jap{t}^{-1-\m}, \quad t\in\re.
\end{aligned}
\end{equation}
Here we have used the fact that $|\y(t)|$ is bounded uniformly in $t$ by virtue of the energy conservation: 
$k(x(t),\x(t))=k(x_0,\x_0)$. This implies the existence of 
\[
\x_\pm =\lim_{t\to\pm\infty} \y(t;x_0,\x_0) =\x_0+\int_0^{\pm\infty} \frac{d\y}{dt}(t;x_0,\x_0)dt.
\]
Moreover, \eqref{eq2-2} also implies 
\begin{equation}\label{eq2-3}
|\y(t)-\x_\pm| =\biggabs{\int_t^{\pm\infty} \frac{d\y}{dt}(t;x_0,\x_0)dt} \leq C\jap{t}^{-\m}, 
\quad \pm t>0.
\end{equation}
Similarly, we have 
\begin{align*}
&\biggabs{\frac{d}{dt} (y_j(t)-t\y_j(t))}\\
&\quad = \biggabs{\sum_m \bigpare{a_{jm}(y(t))-\d_{jm}}\y_m(t)
 +\frac{t}{2}\sum_{m,n}\frac{\pa a_{mn}}{\pa x_j}(y(t))\y_m(t)\y_n(t)} \\
&\quad \leq C\jap{t}^{-\m}, \quad t\in\re,
\end{align*}
and hence 
\begin{equation}\label{eq2-4}
|y(t)-t\y(t)|\leq C\jap{t}^{1-\m}, \quad t\in\re. 
\end{equation}
We compute $\frac{d}{dt}\bigpare{y(t)-\pa_\x\G(t,\y(t))}$ as follows. We have 
\begin{align*}
\frac{\pa\G}{\pa \x_j}(t,\x) 
&= \int_0^t \biggpare{s\frac{\pa k}{\pa x_j}(s\x,\x)+\frac{\pa k}{\pa \x_j}(s\x,\x)}ds  \\
&= \int_0^t \biggpare{\frac{s}{2}\sum_{m,n} \frac{\pa a_{mn}}{\pa x_j}(s\x)\x_m\x_n 
+\sum_m a_{jm}(s\x)\x_m}ds,
\end{align*}
and then 
\begin{align*}
\frac{d}{dt}\biggpare{\frac{\pa\G}{\pa \x_j}(t,\y(t))} 
&= \frac{t}{2}\sum_{m,n} \frac{\pa a_{mn}}{\pa x_j}(t\y)\y_m\y_n +\sum_m a_{jm}(t\y)\y_m \\
&+\sum_i \int_0^t  \biggpare{\frac{s^2}{2}\sum_{m,n} \frac{\pa^2 a_{mn}}{\pa x_j\pa x_i}
(s\y)\y_m\y_n +s\sum_m \frac{\pa a_{mi}}{\pa x_j}(s\y)\y_m \\
&\hspace{2.5cm} +s \sum_m \frac{\pa a_{jm}}{\pa x_i}(s\y)\y_m +a_{ji}(s\y)}ds\times \frac{d\y_i}{dt}.
\end{align*}
We remark that $\y$ in the integrand is $\y(t)$, not $\y(s)$. We also note that the last term 
can be rewritten as 
\begin{align*}
\sum_i \int_0^t a_{ji}(s\y(t))\frac{d\y_i}{dt}(t) ds 
&= t\frac{d\y_j}{dt}(t)+ \sum_i \int_0^t \bigpare{a_{ji}(s\y)-\d_{ji}}ds \times \frac{d\y_i}{dt}(t)\\
&=-\frac{t}{2} \sum_{m,n} \frac{\pa a_{mn}}{\pa x_j}(y(t))\y_m(t)\y_n(t) +O(\jap{t}^{-2\m})
\end{align*}
by using \eqref{eq2-2}. Combining these with Assumption~A, we have 
\[
\begin{aligned}
& \frac{d}{dt}\biggpare{y_j(t)-\frac{\pa\G}{\pa \x_j}(t,\y(t))} \\
&\qquad = -\frac{t}{2}\sum_{m,n} \biggpare{\frac{\pa a_{mn}}{\pa x_j}(y(t))-\frac{\pa a_{mn}}{\pa x_j}(t\y(t))}
\y_m(t)\y_n(t) \\
&\qquad \quad + \sum_m \bigpare{a_{jm}(y(t))-a_{jm}(t\y(t))}\y_m(t) +O(\jap{t}^{-2\m}). 
\end{aligned}
\]
Using Assumption~A again with \eqref{eq2-4}, we obtain 
\[
\biggabs{\frac{d}{dt}\biggpare{y(t)-\frac{\pa\G}{\pa \x}(t,\y(t))}}\leq C\jap{t}^{-2\m}, \quad t\in\re. 
\]
Since $2\m>1$, this implies the existence of 
\[
x_\pm =\lim_{t\to\pm\infty} \biggpare{y(t)-\frac{\pa\G}{\pa \x}(t,\y(t))}.
\]
The assertion: $W_\pm^{cl}$ is a diffeomorphism can be proved by the standard ODE method, 
once we have the integrability. 
\end{proof}

In the proof of Theorem~2, we actually consider the high energy asymptotics of $(x(t),\x(t))$ for fixed $t$. 
If $V=0$, i.e., for $(y(t),\y(t))$, we have the scaling property: 
\[
(y(t;x_0,\l\x_0),\y(t;x_0,\l\x_0))
=(y(\l t;x_0,\x_0),\l \y(\l t;x_0,\x_0))
\]
for any $\l>0$. Hence we learn
\[
\lim_{\l\to\infty} {\frac1\l \y(t;x_0,\l\x_0)} 
=\lim_{\l\to\infty} \y(\l t;x_0,\x_0) =\x_\pm 
\]
if $\pm t>0$. Since 
\[
\G(t,\l\x) =\int_0^t k(s\l\x,\l\x)ds =\int_0^{\l t} \l^2 k(\s\x,\x)\frac{d\s}{\l} =\l \G(\l t,\x),
\]
we have 
\[
\frac{\pa \G}{\pa \x}(t,\l\x)=\frac{\pa \G}{\pa\x}(\l t,\x).
\]
Using this, we also learn 
\begin{align*}
&\lim_{\l\to\infty} \biggpare{y(t;x_0,\l\x_0)-\frac{\pa\G}{\pa\x}(t,\y(t;x_0,\l\x_0))} \\
&\quad = \lim_{\l\to\infty} \biggpare{y(\l t;x_0,\x_0)-\frac{\pa\G}{\pa\x}(\l t,\y(t;x_0,\x_0))}
=x_\pm
\end{align*}
if $\pm t>0$. We have similar high energy asymptotics for $(x(t;x_0,\l\x_0), \x(t;x_0,\l\x_0))$ 
if we replace $\G(t,\x)$ by $\F(t,\x)$: 

\begin{thm}
Let $(x_0,\x_0)\in\re^d\times\re^d$, $\x_0\neq 0$, and set 
$(x(t;x_0,\l\x_0), \x(t;x_0,\l\x_0))=\exp(tH_p)(x_0,\l\x_0)$ as above. Then 
\begin{align*}
& x_\pm = \lim_{\l\to\infty} \biggpare{x(t;x_0,\l\x_0)-\frac{\pa\F}{\pa\x}(t,\y(t;x_0,\l\x_0))},\\
& \x_\pm =\lim_{\l\to\infty} \frac1\l \x(t;x_0,\l\x_0)
\end{align*}
if $\pm t>0$, where $(x_0,\x_0)= W_\pm^{cl}(x_\pm,\x_\pm)$.
The convergence is locally uniform with their derivatives. 
\end{thm}

\begin{proof}
The proof is similar to that of Proposition~1, but slightly more involved. 
We fix $T>0$ and consider $(x(t;x_0,\l\x_0), \x(t;x_0,\l\x_0))$ with $|t|\leq T$. 

For $\l>0$, we set
\[
x^\l(t;x_0,\x_0) = x\Bigpare{\frac{t}{\l};x_0,\l\x_0}, \quad 
\x^\l(t;x_0,\x_0) = \frac1\l\, \x\Bigpare{\frac{t}{\l};x_0,\l\x_0}.
\]
The it is easy to check 
\[
(x^\l(t;x_0,\x_0), \x^\l(t;x_0,\x_0)) = \exp(tH_{p^\l})(x_0,\x_0),
\]
where 
\[
p^\l(x,\x)= \frac{1}{\l^2} p(x,\l\x) =\frac12 \sum_{m,n=1}^d a_{mn}(x)\x_m\x_n+\frac{1}{\l^2}V(x).
\]
We can show, as well as \eqref{eq2-1}, 
\begin{equation}\label{eq2-5}
|x^\l(t;x_0,\x_0)|\geq c|t|-C, \quad |t|\leq \l T, 
\end{equation}
uniformly for $\l\geq \l_0\gg 0$ (\cite{Na2}, Proposition~2.6). The constants in the following 
proof are independent of such large $\l\geq \l_0\gg 0$. 
The proof of \eqref{eq2-5} relies on the idea that 
$\l^{-2}V(x)$ has significantly smaller effect than the kinetic energy part $k(x,\x)$ if $\l$ is sufficiently large. 
We refer \cite{Na2} for the complete proof. 

Then, as well as \eqref{eq2-2}, we have 
\begin{align*}
\biggabs{\frac{d}{dt} \x_j^\l(t)} 
&=\biggabs{\frac12 \sum_{m,n} \frac{\pa a_{mn}}{\pa x_j}(x^\l)\x^\l_m\x^\l_n 
+\frac{1}{\l^2}\frac{\pa V}{\pa x_j}(x^\l)}\\
&\leq C \jap{t}^{-1-\m} + C\l^{-2}\jap{t}^{1-\m} \\
&\leq C\jap{t}^{-1-\m}
\end{align*}
if $|t|\leq \l T$. On the other hand, by the continuity of the solution to linear ODEs with respect to the coefficients, 
we learn 
\[
\exp(tH_{p^\l})(x_0,\x_0) \to \exp(tH_k)(x_0,\x_0) \quad \text{as }\l\to\infty
\]
for each fixed $t\in\re$. The convergence holds also for the derivatives. 
Then we can apply the dominated convergence theorem to show 
\[
\x^\l(\l t) =\x_0+\int_0^{\l t} \frac{d\x^\l}{dt}(s) ds \ \to \ 
\x_0 +\int_0^{\pm\infty}\frac{d\y}{dt}(s)ds =\x_\pm
\]
as $\l\to\infty$, where $0<\pm t\leq T$. This proves the second statement of the theorem. 
We also have (as well as \eqref{eq2-3}), 
\[
|\x^\l(t)-\x_\pm|\leq C\jap{t}^{-\m}, \quad 0<\pm t \leq \l T.
\]
Then we compute 
\begin{align*}
\biggabs{\frac{d}{dt}\biggpare{y_j^\l(t)-t \x_j^\l(t)}}
&=\biggabs{\sum_m \bigpare{a_{jm}(x^\l)-\d_{jm}}\x^\l \\
&\qquad +\frac{t}{2}\sum_{m,n} \frac{\pa a_{mn}}{\pa x_j}(x^\l)\x^\l_m\x^\l_n 
+\frac{t}{\l^2}\frac{\pa V}{\pa x_j}(x^\l)}\\
&\leq C\jap{t}^{-\m} + C\l^{-2}\jap{t}^{2-\m} \\
&\leq \jap{t}^{-\m}, \qquad\text{if }|t|\leq T, 
\end{align*}
and hence 
\begin{equation}\label{eq2-6}
|x^\l(t)-t\x^\l(t)|\leq C\jap{t}^{1-\m}, \quad |t|\leq \l T.
\end{equation}
We set
\[
\F^\l(t,\x)= \int_0^t \biggpare{k(s\x,\x)+\frac{1}{\l^2} V^{(L)}(s\x)}ds
\]
so that 
\[
\frac{\pa \F^\l}{\pa \x}(\l t,\x) = \frac{\pa \F}{\pa\x}(t,\l\x).
\]
Then we have 
\[
\frac{\pa \F^\l}{\pa \x_j}(t,\x) =\int_0^t \biggpare{\frac{s}{2}\sum_{m,n} \frac{\pa a_{mn}}{\pa x_j}
(s\x)\x_m\x_n +\frac{s}{\l^2}\frac{\pa V^{(L)}}{\pa x_j}(s\x)+\sum_m a_{jm}(s\x)\x_m}ds,
\]
and then 
\begin{align*}
&\frac{d}{dt}\biggpare{\frac{\pa \F^\l}{\pa \x_j}(t,\x^\l(t))}\\
&=\frac{t}{2} \sum_{m,n} \frac{\pa a_{mn}}{\pa x_j}(t\x^\l)\x^\l_m\x^\l_n 
+ \frac{t}{\l^2}\frac{\pa V^{(L)}}{\pa x_j}(s\x^\l) +\sum_m a_{jm}(t\x^\l)\x^\l \\
&\quad +\sum_i \int_0^t \biggpare{\frac{s^2}{2} \sum_{m,n} \frac{\pa^2 a_{mn}}{\pa x_i\pa x_j}(s\x^\l)
\x^\l_m\x^\l_n + s\sum_m \frac{\pa a_{mi}}{\pa x_j}(s\x^\l)\x^\l_m \\
& \qquad  +\frac{s^2}{\l^2}\frac{\pa^2 V^{(L)}}{\pa x_i\pa x_j}(s\x^\l) 
+s \sum_m \frac{\pa a_{jm}}{\pa x_i}(s\x^\l)\x^\l_m +a_{ji}(s\x^\l)} ds \times \frac{d\x^\l_i}{dt} . 
\end{align*}
We recall
\begin{align*}
&\frac{d x^\l_j}{d t}(t) = \sum_m a_{jm}(x^\l(t))\x^\l_m(t), \\
&\frac{d\x^\l_j}{dt}(t) = -\frac12 \sum_{m,n}\frac{\pa a_{mn}}{\pa x_j}(x^\l(t))\x^\l_m(t)\x^\l_n(t)\\
&\hspace{1.8cm} -\frac{1}{\l^2}\frac{\pa V^{(L)}}{\pa x_j}(x^\l(t)) - \frac{1}{\l^2}\frac{\pa V^{(S)}}{\pa x_j}(x^\l(t)). 
\end{align*}
Combining these, we obtain as in the proof of Proposition~1, 
\begin{align*}
&\frac{d}{dt}\biggpare{x^\l_j(t)-\frac{\pa \F^\l}{\pa \x_j}(t,\x^\l(t))}\\
&\quad = -\frac{t}{2}\sum_{m,n} \biggpare{\frac{\pa a_{mn}}{\pa x_j}(x^\l)-\frac{\pa a_{mn}}{\pa x_j}(t\x^\l)}
\x^\l_m\x^\l_n 
- \frac{t}{\l^2} \biggpare{\frac{\pa V^{(L)}}{\pa x_j}(x^\l) -\frac{\pa V^{(L)}}{\pa x_j}(t\x^\l) } \\
&\qquad +\sum_m \bigpare{a_{jm}(x^\l)-a_{jm}(t\x^\l)}\x^\l_m 
+ O\Bigpare{\jap{t}^{-2\m} +\l^{-2}\jap{t}^{2-2\m}+\l^{-2}\jap{t}^{2-\n}}. 
\end{align*}
We recall \eqref{eq2-6}, and using Assumption~C, we have 
\[
\biggabs{\frac{d}{dt}\biggpare{x^\l(t)-\frac{\pa \F^\l}{\pa \x}(t,\x^\l(t))}}\leq C\jap{t}^{-2\m'}
\]
if $|t|\leq \l T$, where $\m'=\min(\m,\n/2)>1/2$. Then, again observing 
\[
\frac{d}{dt} x^\l(t) \to \frac{d}{dt} y(t), \quad 
\frac{d}{dt}\biggpare{\frac{\pa \F^\l}{\pa\x}(t,\x^\l(t))}\to 
\frac{d}{dt}\biggpare{\frac{\pa \G}{\pa \x}(t,\y(t))}
\]
as $\l\to\infty$ for each $t$, and using the dominated 
convergence theorem, we conclude
\[
x^\l(\l t)-\frac{\pa \F^\l}{\pa \x}(\l t,\x^\l(\l t))\to x_\pm
\quad \text{as }\l\to\infty,
\]
if $0<\pm t\leq T$. The first statement of the theorem follows immediately from this. 
The last claim can be proved using the standard ODE method. 
\end{proof}

Now we prepare several estimates for the next section. 

\begin{lem}
Let $T>0$. Then for any $\a\in\ze_+^d$ there is $C_\a>0$ such that 
\[
\Bigabs{\pa_\x^\a\bigpare{\F(t,\x)-\tfrac{1}{2}t|\x|^2}}\leq C_\a |t| \jap{\x}^{2-\m-|\a|}, 
\quad \x\in\re^d, |t|\leq T.
\] 
In particular, 
\[
\Bigabs{\pa_\x^\a \bigpare{\pa_\x \F(t,\x)-t\x}}\leq C_\a |t| \jap{\x}^{1-\m-|\a|}, 
\quad \x\in\re^d, |t|\leq T. 
\]
\end{lem}

\begin{proof}
We suppose $t\geq 0$. By the definition, we have 
\[
\F(t,\x)-\tfrac{1}{2}t |\x|^2 = \frac12 \int_0^t \sum_{m,n=1}^d \bigpare{a_{mn}(s\x)-\d_{mn}} \x_m\x_n ds + \int_0^t V^{(L)}(s\x) ds, 
\]
and hence 
\begin{align*}
&\Bigabs{\pa_\x^\a\bigpare{\F(t,\x)-\tfrac12 t|\x|^2}} \\
&\leq \frac12 \int_0^t \sum_{m,n} \bigabs{\pa_\x^\a\bigbra{(a_{mn}(s\x)-\d_{mn})\x_m\x_n}}ds
+\int_0^t \bigabs{\pa_\x^\a\bigpare{V^{(L)}(s\x)}} ds \\
&\leq C \sum_{j=0}^{|\a|} \int_0^t s^j \jap{s\x}^{-\m-j }|\x|^{2-(|\a|-j)} ds 
+C\int_0^t s^{|\a|} \jap{s\x}^{2-\m-|\a|} ds\\
&=C \sum_{j=0}^{|\a|} \int_0^{t|\x|} \s^j \jap{\s}^{-\m-j} |\x|^{1-|\a|} d\s 
+C\int_0^{t|\x|} \s^{|\a|} \jap{\s}^{2-\m-|\a|}  |\x|^{-1-|\a|} d\s\\
&\leq C \int_0^{t|\x|}  \jap{\s}^{-\m}d\s\cdot  |\x|^{1-|\a|} 
+C\int_0^{t|\x|}  \jap{\s}^{2-\m} d\s \cdot |\x|^{-1-|\a|} \\
&\leq C|t\x| \jap{t|\x|}^{-\m}|\x|^{1-|\a|} + C|t\x| \jap{t|\x|}^{2-\m}|\x|^{-1-|\a|} \\
&\leq C|t| \jap{\x}^{2-\m-|\a|}, \qquad\text{if } |t|\leq T,\ |\x|\geq 1.  
\end{align*}
The case $t<0$ is handled similarly. 
\end{proof}

We then set
\[
z(t;x_0,\x_0) =x(t;x_0,\x_0)-\pa_\x\F(t;\x(t;x_0,\x_0)),
\]
and consider the time evolution:
\[
t\mapsto (z(t;x_0,\x_0),\x(t;x_0,\x_0)), \quad |t|\leq T. 
\]
We recall 
\[
x_\pm = \lim_{\l\to\infty} z(t;x_0,\l\x_0), \quad \text{when }0<\pm t\leq T, 
\]
and in particular, $\bigset{z(t;x_0,\l\x_0)}{|t|\leq T, \l\geq \l_0}$ is bounded in $\re^d$,
provided $\l_0$ is sufficiently large. We set
\[
\ell(t;z,\x)= p\bigpare{z+\pa_\x\F(t,\x),\x} -\frac{\pa\F}{\pa t}(t,\x).
\]
Then we can show that $(z(t),\x(t))$ is the Hamilton flow generated by the time-dependent 
Hamiltonian $\ell(t;z,\x)$: 

\begin{lem}
$(z(t),\x(t))=(z(t;x_0,\x_0),\x(t;x_0,\x_0))$ is the solution to 
\[
\frac{d z}{dt}(t) =\frac{\pa\ell}{\pa\x}(t;z(t),\x(t)), \quad
\frac{d\x}{dt}=-\frac{\pa\ell}{\pa z}(t;z(t),\x(t)), 
\]
with the initial condition: $z(0)=x_0$, $\x(0)=\x_0$. 
\end{lem}

\begin{proof}
The second equation and the initial conditions are easy to confirm. 
By the definitions, we have 
\[
\frac{\pa\ell}{\pa\x_j}(t;z,\x) =\frac{\pa p}{\pa\x_j}\bigpare{z+\pa_\x\F,\x}
+\sum_{m=1}^d \frac{\pa^2\F}{\pa\x_j\pa\x_m}\frac{\pa p}{\pa x_m}\bigpare{z+\pa_\x\F,\x}
- \frac{\pa^2\F}{\pa\x_j\pa t},
\]
and 
\[
\frac{d z_j}{dt}(t)=  \frac{dx_j}{dt}
-\frac{\pa^2\F}{\pa\x_j\pa t} -\sum_{m=1}^d \frac{\pa^2\F}{\pa \x_j\pa\x_m}\frac{d\x_m}{dt}. 
\]
On the other hand, by the Hamilton equation for $(x(t),\x(t))$, we have 
\[
\frac{dx_j}{dt} = \frac{\pa p}{\pa\x_j}\bigpare{z+\pa_\x\F,\x}, \quad 
\frac{d\x_m}{dt}= -\frac{\pa p}{\pa x_m}\bigpare{z+\pa_\x\F,\x}, 
\]
and we conclude the first equation of the lemma combining them. 
\end{proof}

By the definition, we easily see 
\[
\frac{\pa\F}{\pa t}(t,\x)= p^{(L)}(t\x,\x), 
\]
and hence we can write
\begin{align*}
\ell(t;z,\x) &= p\bigpare{z+\pa_\x\F(t,\x),\x}-p^{(L)}(t\x,\x) \\
&= \frac12 \sum_{m,n=1}^d \bigpare{a_{mn}\bigpare{z+\pa_\x\F(t,\x)}-a_{mn}(t\x)}\x_m\x_n \\
&\quad+ \bigpare{V^{(L)}\bigpare{z+\pa_\x\F(t,\x)}-V^{(L)}(t\x)} 
+V^{(S)}\bigpare{z+\pa_\x\F(t,\x)}. 
\end{align*}
Combining this with Lemma~7, we obtain: 

\begin{lem}
Let $K\subset \re^d$ be a bounded domain, and $\a,\b\in\ze_+^d$. Then there is $C_{K\a\b}>0$ such that 
\[
\Bigabs{\pa_z^\a\pa_\x^\b \ell(t;z,\x)}\leq C_{K\a\b}\jap{\x}^{1-\c-|\b|}, 
\quad z\in K, \x\in\re^d, |t|\leq T, 
\]
where $\c=\min(2\m-1,\n-1)>0$. 
\end{lem}


\section{Proof of the main theorem}

The proof of Theorem~2 is analogous to the proof of \cite{Na2} Theorem~1.2, given the estimates 
on the classical flow in Section~2. We sketch it for the completeness, and give somewhat formal proof. 

For a symbol $a\in S^m_{1,0}$, we quantize it using the Weyl calculus (\cite{Ho} Section~18.5): 
\[
a^W\!(x,D_x)u(x) =(2\pi)^{-d} \iint e^{i(x-y)\cdot\x} a\Bigpare{\frac{x+y}{2},\x} u(y) dyd\x, 
\quad u\in\mathcal{S}(\re^d). 
\]
By direct computations, it is easy to see 
\begin{align*}
p^W\!(x,D_x) &= -\frac18 \sum_{m,n} \biggpare{\frac{\pa^2}{\pa x_m\pa x_n}a_{mn}(x) 
+2\frac{\pa}{\pa x_m} a_{mn}(x)\frac{\pa}{\pa x_n} \\
&\hspace{5cm} +a_{mn}(x)\frac{\pa^2}{\pa x_m\pa x_n}} +V(x)\\
&= H-\frac18 \sum_{m,n} \frac{\pa^2 a_{mn}}{\pa x_m\pa x_n}(x), 
\end{align*}
where $p(x,\x)= \frac12 \sum a_{mn}(x)\x_m\x_n+V(x)$. Hence, by replacing $V$ in $p(x,\x)$ by 
$V+\frac18\sum\frac{ \pa^2 a_{mn}}{\pa x_m \pa x_n} (x)$, we may consider 
$H=p^W\!(x, D_x)$. 

We are interested in the behavior of 
\[
v(t)= e^{i\F(t,D_x)} e^{-itH} u_0, \quad t\in \re.
\]
For $u_0\in\mathcal{S}(\re^d)$, we can differentiate $v(t)$ in $t$, and we have 
\begin{align*}
\frac{d}{dt} v(t) &= i e^{i\F(t,D_x)} \biggpare{\frac{\pa \F}{\pa t}(t,D_x)-H} e^{-itH}u_0\\
&= -i L(t) v(t), 
\end{align*}
where 
\[
L(t) =e^{i\F(t,D_x)} H e^{-i\F(t,D_x)} -\frac{\pa \F}{\pa t}(t,D_x).
\]
We note 
\begin{align*}
e^{i\F(t,D_x)} x e^{-i\F(t,D_x)} 
&= \mathcal{F}^* \Bigbrac{e^{i\F(t,\x)} i\pa_\x  e^{-i\F(t,\x)} }\mathcal{F}\\
&=\mathcal{F}^* \Bigbrac{i\pa_\x +\pa_\x\F(t,\x)}\mathcal{F} 
= x+\pa_\x\F(t,D_x), 
\end{align*}
where $\mathcal{F}$ is the Fourier transform. Hence, we expect
\[
e^{i\F(t,D_x)} p^W\!(x,D_x) e^{-i\F(t,D_x)} \sim p^W\!\bigpare{x+\pa_\x\F(t,D_x),D_x}
\]
in some sense. In fact, combining Lemma~7 with \cite{Na2} Lemma~3.1, we have the following: 

\begin{lem}
$e^{i\F(t,D_x)} H e^{-i\F(t,D_x)}$ is a pseudodifferential operator with the symbol in $S^2_{1,0}$. 
Moreover, if we set $\tilde p(t,x,\x) =p(x+\pa_\x\F(t,\x),\x)$, then 
\[
e^{i\F(t,D_x)} H e^{-i\F(t,D_x)} -\tilde p^W\!(t,x,D_x) =r^W\!(t,x,D_x)
\]
with $r\in S^0_{1,0}$, i.e., for any $\a,\b\in\ze_+^d$, $K\Subset\re^d$, there is $C_{\a\b K}>0$ 
such that 
\[
\Bigabs{\pa_x^\a\pa_\x^\b r(t,x,\x)}\leq C_{\a\b K}\jap{\x}^{-|\b|}, \quad x\in K, \x\in\re^d, |t|\leq T.
\]
\end{lem}

Thus, by recalling the definition of $\ell(t;x,\x)$ in Section~2, we learn that the principal symbol 
of $L(t)$ is given by $\ell(t;x,\x)$. This is consistent with the fact that $e^{i\F(t,D_x)} e^{-itH}$ 
is the quantization of the classical flow: $(x_0,\x_0)\mapsto (z(t),\x(t))$, of which $\ell(t;x,\x)$ 
is the Hamiltonian. 

In order to analyze the microlocal singularities, we use the semiclassical type characterization 
of the wave front set: Let $(x_0,\x_0)\in\re^d\times\re^d$, $\x_0\neq 0$, and $u\in\mathcal{S}'(\re^d)$. 
$(x_0,\x_0)\notin \WF(u)$ if and only if there is $a\in C_0^\infty(\re^d\times \re^d)$ such that 
$a(x_0,\x_0)\neq 0$ and 
\begin{equation}\label{eq3-1}
\bignorm{a^W\!(x,\l^{-1} D_x) u}\leq C_N \l^{-N}, \quad \l\gg 0, 
\end{equation}
with any $N\in\ze_+$ (see, e.g., \cite{Ma} Section~2.9). 

Let $(x_0,\x_0)\in\re^d\times\re^d$, $\x_0\neq 0$, be fixed, and suppose 
$a_0\in C_0^\infty(\re^d\times\re^d)$ is supported in a small neighborhood of $(x_0,\x_0)$, 
for example, $B_\d(x_0,\x_0)=\bigset{(x,\x)}{|x-x_0|^2+|\x-\x_0|^2<\d^2}$. We set
\[
A(t)= e^{i\F(t,D_x)} e^{-itH} A_0\,  e^{itH} e^{-i\F(t,D_x)}, 
\quad A(0)= A_0= a_0^W\!(x,\l^{-1} D_x), 
\]
and consider the time evolution of $A(t)$. In the weak sense on $\mathcal{S}(\re^d)$, we can 
compute the derivative of $A(t)$ in $t$, and we obtain the Heisenberg equation: 
\begin{equation}\label{eq3-2}
\frac{d}{dt} A(t) = -i[L(t), A(t)], \quad A(0)= a_0^W\!(x, \l^{-1} D_x).
\end{equation}
We construct an asymptotic solution to this equation as $\l\to\infty$, following the 
Egorov type argument. The corresponding transport equation is given by 
\[
\frac{\pa}{\pa t}a(t,x,\x)= -\{\ell, a\}(t,x,\x) 
=-\sum_{m=1}^d \biggpare{\frac{\pa\ell}{\pa\x_m}\frac{\pa a}{\pa x_m}-
\frac{\pa\ell}{\pa x_m}\frac{\pa a}{\pa \x_m}}
\]
with the initial condition $a(0,x,\x)= a_0(x,\l^{-1}\x)$. We denote 
\[
\S_t: (x_0,\x_0)\mapsto (z(t),\x(t)), \quad a_0^\l(x,\x)=a_0(x,\l^{-1}\x).
\]
Then the solution to the transport equation is given by 
\[
\tilde a_0(t;x,\x) =(a_0^\l\circ \S_t^{-1})(x,\x), 
\]
and $\tilde a_0(t,\cdot,\cdot)$ is supported in $\S_t(\supp(a_0^\l))$. 
We note $a_0^\l$ is bounded in $S^0_{1,0}$, uniformly in $\l\in [1,\infty)$. This also 
implies $\tilde a_0(t,\cdot,\cdot)$ is uniformly bounded in $S^0_{1,0}$, provided $|t|\leq T$. 
Combining this observation with Lemma~9, we learn that 
\[
-i[L(t),A_0(t)] +\{\ell,\tilde a_0\}^W\!(x,D_x) =r_0^W\!(t;x,D_x)
\]
with $r_0\in S^{-1}_{1,0}$ uniformly in $\l$. Moreover, by the asymptotic expansion, 
we learn $r_0(t;\cdot,\cdot)$ is supported in $\S_t(\supp(a_0^\l))$ modulo $O(\l^{-\infty})$ terms. 

Following the standard Egorov type argument with the help of the scaling argument in Section~2, we can 
construct an asymptotic solution $\tilde a(t;x,\x)$ with the following properties (see \cite{Na2} 
Propostion~3.2): 
\begin{enumerate}
\renewcommand{\theenumi}{\roman{enumi}}
\renewcommand{\labelenumi}{\theenumi )}
\item $\tilde a(0;x,\x)=a_0^\l(x,\x)=a_0(x,\l^{-1}\x)$. 
\item $\tilde a(t;\cdot,\cdot)$ is supported in  $\S_t(\supp(a_0^\l))$. 
\item For any $\a,\b\in\ze_+^d$, there is $C_{\a\b}>0$ such that 
\[
\bigabs{\pa_x^\a\pa_\x^\b \tilde a(t;x,\x)}\leq C_{\a\b}\l^{-|\b|}, 
\quad |t|\leq T, x,\x\in\re^d, \l\geq 1. 
\]
\item The principal symbol of $\tilde a(t;x,\x)$ is given by $a_0^\l\circ \S_t^{-1}$, 
i.e., 
\[
\bigabs{\pa_x^\a\pa_\x^\b\bigpare{\tilde a(t;x,\x)-(a_0^\l\circ \S_t^{-1})(x,\x)}}
\leq C_{\a\b}\l^{-1-|\b|}
\]
for $|t|\leq T$, $x,\x\in\re^d$, $\l\geq 1$. 
\item $\tilde A(t)=\tilde a^W\!(t;x,D_x)$ satisfies the Heisenberg equation \eqref{eq3-2} asymptotically, 
i.e, 
\[
\biggnorm{\frac{d}{dt}\tilde A(t) +i\Bigbrac{L(t),\tilde A(t)}}\leq C_N \l^{-N}, \quad\l\geq 1, 
\]
for any $N\in\ze_+$ with some $C_N>0$. 
\end{enumerate}

These properties imply 
\[
\biggnorm{\frac{d}{dt} \Bigpare{ e^{itH} e^{-i\F(t,D_x)}\tilde A(t) e^{i\F(t,D_x)} e^{-itH}}}
\leq C_N \l^{-N}, \quad |t|\leq T, 
\]
and hence 
\begin{equation}\label{eq3-3}
\Bignorm{ e^{itH} e^{-i\F(t,D_x)}\tilde A(t) e^{i\F(t,D_x)} e^{-itH}-a_0^W\!(x,\l^{-1}D_x)} \leq C_N \l^{-N}
\end{equation}
if $|t|\leq T$. This is equivalent to 
\begin{equation}\label{eq3-4}
\Bignorm{\tilde A(t)- e^{i\F(t,D_x)} e^{-itH} a_0^W\!(x,\l^{-1}D_x) e^{itH} e^{-i\F(t,D_x)}} 
\leq C_N\l^{-N}. 
\end{equation} 

On the other hand, if we write 
\begin{align*}
\S_t^\l(x,\x)&= (z(\l^{-1}t;x,\l\x),\l^{-1}\x(\l^{-1}t;x,\l\x) ) \\
&= \bigpare{x^\l(t;x,\x)-\pa_\x \F^\l(t,\x^\l(t;x,\x)),\x^\l(t;x,\x)},
\end{align*}
then 
\[
\bigpare{a_0^\l\circ \S_t^{-1}}(x,\l\x) =\bigpare{a_0\circ (\S_{\l t}^\l)^{-1}}(x,\x).
\]
We note $(\S_{\l t}^\l)^{-1}$ converges to $W_\pm^{cl}$ as $\l\to\infty$ when $\pm t>0$ locally uniformly, 
with its derivatives (Theorem~6). Hence we learn that $\bigpare{a_0^\l\circ \S_t^{-1}}(x,\l\x)$ converges to 
$\bigpare{a_0\circ W_\pm^{cl}}(x,\x)$ uniformly in $(x,\x)$ with its derivatives, 
and that the support of $(a_0^\l\circ \S_t^{-1})(x,\l\x)$ also converges to the support of $a_0\circ W_\pm^{cl}$. 
This implies $(a_0^\l\circ \S_t^{-1})^W\!(x,D_x)$ converges to $(a_0\circ W_\pm^{cl})^W\!(x,\l^{-1} D_x)$ as 
$\l\to\infty$ including their microlocal support properties. 

If $(x_0,\x_0)\notin \WF(u_0)$, and $a_0$ is supported in a small neighborhood of $(x_0,\x_0)$ 
such that \eqref{eq3-1} holds, then \eqref{eq3-3} implies 
\[
\bignorm{ \tilde A(t) e^{i\F(t;D_x)}e^{-itH}u_0} \leq C_N \l^{-N}.
\]
Since the principal symbol of $\tilde A(t)$ is $(a_0^\l\circ \S_t^{-1})(x,\x)$, which is very close to 
$(a_0\circ W_\pm^{cl})(x,\l\x)$, this implies $(W_\pm^{cl})^{-1}(x_0,\x_0)\notin \WF\bigpare{e^{i\F(t;D_x)}e^{-itH}u_0}$. 

Similarly, if $(W_\pm^{cl})^{-1}(x_0,\x_0)\notin \WF\bigpare{e^{i\F(t;D_x)}e^{-itH}u_0}$, then we can conclude
$(x_0,\x_0)\notin \WF(u_0)$ using \eqref{eq3-4}. 

The above formal argument can be easily justified as in \cite{Na2} Section~3.2, and 
Theorem~2 is proved. \qed 

\section{Asymptotically homogeneous potentials} 

Here we consider the case $a_{mn}(x)=\d_{mn}$, and $V(x)$ is asymptotically homogeneous of 
order $\b\in [1,3/2)$. 

\begin{proof}[Proof of Theorem~4]
Suppose $V^{(L)}(x)=|x|V^{(L)}(\hat x)$, $\hat x=x/|x|$, if $|x| \geq 1$, and let $t>0$. Then if $|\x|\geq t^{-1}$, 
we have 
\begin{align*}
\int_0^t V^{(L)}(s\x)ds 
&= \int_0^t s|\x|V^{(L)}(\hat\x)ds +\int_0^{1/|\x|} \bigpare{V^{(L)}(s\x)-s|\x| V^{(L)}(\hat \x)}ds\\
&= \frac{t^2}{2}|\x|V^{(L)}(\hat\x)+R(t,\x). 
\end{align*}
$R(t,\x)$ can be computed as 
\[
R(t,\x)=\int_0^1 \bigpare{V^{(L)}(\s\hat\x)-\s V^{(L)}(\hat\x)}|\x|^{-1}d\s, 
\]
and hence for any $\a\in\ze_+^d$, 
\[
\bigabs{\pa_\x^\a R(t,\x)}\leq C_\a |\x|^{-1-|\a|}, \quad |\x|\geq t^{-1}. 
\]
Hence, if we set 
\[
F(t,\x)=\int_0^t V^{(L)}(s\x)ds -\frac{t^2}{2} V^{(L)}(\x), 
\]
then for any fixed $t\neq 0$, $F(t,\x)\in S^{-1}_{1,0}$, i.e., for any $\a\in\ze_+^d$, 
\[
\bigabs{\pa_\x^\a F(t,\x)}\leq C_\a \jap{\x}^{-1-|\a|}, \quad \x\in\re^d. 
\]
This implies $e^{iF(t,\x)}\in S^0_{1,0}$, and it is obviously elliptic. In particular, we have 
\[
\WF(e^{iF(t,D_x)}u)=\WF(u), \quad u\in L^2(\re^d). 
\]
Now we note 
\[
\F(t,\x)=\frac{t}{2}|\x|^2 +\int_0^t V^{(L)}(s\x)ds =\frac{t}{2}|\x|^2 +\frac{t^2}{2}V^{(L)}(\x) +F(t,\x). 
\]
Combining these with Theorem~2, we learn 
\begin{align*}
\WF(u)&= \WF\bigpare{e^{i\F(t,D_x)}e^{-itH}u} \\
&= \WF\bigpare{e^{iF(t,D_x)} e^{i(t^2/2)V^{(L)}(D_x)} e^{itH_0} e^{-itH}u} \\
&= \WF\bigpare{e^{i(t^2/2)V^{(L)}(D_x)} e^{itH_0} e^{-itH}u}.
\end{align*}
Since $V^{(L)}(\x)$ is homogeneous of order 1,  $e^{i(t^2/2)V^{(L)}(D_x)}$ is a Fourier integral operator 
with the associated canonical transform: 
\[
S_{t^2/2}^+\ :\ (x,\x)\to \Bigpare{x+\frac{t^2}{2}\pa_x V^{(L)}(\hat\x),\x}
\]
and hence 
\begin{equation}\label{eq4-1}
\WF\bigpare{e^{i(t^2/2)V^{(L)}(D_x)}v } =S^+_{t^2/2}(\WF(v))
\end{equation}
(see, e.g., \cite{Ho} Chapter XXV, \cite{Ta} Section~VIII.5). 
Thus we have 
\[
\WF(u)= S_{t^2/2}^+\bigpare{\WF\bigpare{e^{itH_0} e^{-itH} u}},
\]
and it is equivalent to 
\[
\WF\bigpare{e^{itH_0} e^{-itH} u}= S^+_{(-t^2/2)}(\WF(u)).
\]
If $t<0$, then 
\[
\int_0^t V^{(L)}(s\x)ds =-\int_0^{|t|} V^{(L)}(-s\x)ds, 
\]
and we replace $V^{(L)}(\hat\x)$ by $V^{(L)}(-\hat\x)$, and we change the direction of the shift
of obtain $S_{t^2/2}^-$ in the statement. 
\end{proof}

\begin{rem}
The propery \eqref{eq4-1} can also be proved using the propagation of singularities theorem for hyperbolic equations. 
In fact, $e^{i\s V^{(L)}(D_x)}u_0$ is the solution to the hyperbolic evolution equation: 
\[
\frac{\pa}{\pa \s} u(\s) =iV^{(L)}(D_x) u(\s), \quad u(0)=u_0, 
\]
and the claim \eqref{eq4-1} follows from, for example, the Egorov theorem (see, e.g., \cite{Ta} Section~VIII.2). 
\end{rem}

\begin{proof}[Proof of Theorem 5]
We suppose $t>0$, and $V^{(L)}(x)= |x|^\b V^{(L)}(\hat\x)$ for $|x|\geq 1$ with $1<\b<3/2$. 
By the same computation as in the proof of Theorem~4, we have 
\[
\int_0^t V^{(L)}(s\x) ds =\frac{t^{1+\b}}{1+\b} |\x|^\b V^{(L)}(\hat\x) +R(t,\x)
\]
with 
\[
\bigabs{\pa_\x^\a R(t,\x)}\leq C_\a |\x|^{-1-|\a|}, \quad |\x|\geq t^{-1}
\]
for any $\a\in\ze_+^d$. Thus we we learn, as well: 
\[
\F(t,\x) = \frac{t}{2}|\x|^2 +\frac{t^{1+\b}}{1+\b} V^{(L)}(\x) +F(t,\x)
\]
with $F(t,\x)\in S^{-1}_{1,0}$, and hence 
\[
\WF(u)= \WF\bigpare{e^{i\s V^{(L)}(D_x)} e^{itH_0} e^{-itH} u}
\]
where $\s=\dfrac{t^{1+\b}}{1+\b}$. This implies 
\begin{equation}\label{eq4-2}
\WF\bigpare{e^{itH}u} =\WF\bigpare{e^{i\s V^{(L)}(D_x)} e^{itH_0} u}.
\end{equation}
Since $V(\x)$ is homogeneous of order $\b>1$, 
we can prove $e^{i\s V^{(L)}(D_x)}$ has the diffusive property: 

\begin{lem}
Let $N\in\ze_+$ and let $\s\neq 0$.  Then there is $C_N$ such that 
\begin{equation}\label{eq4-3}
\bignorm{\jap{x}^{-N} e^{i\s V^{(L)}(D_x)}u}_{H^s} \leq C_N \bignorm{\jap{x}^N u}, 
\quad u\in L^{2,\infty}(\re^d),
\end{equation}
where $s=(\b-1)N$. In particular, 
$e^{i\s V^{(L)}(D_x)}u\in C^\infty(\re^d)$ if $u\in L^{2,\infty}(\re^d)$. 
\end{lem}

\begin{proof}
For simplicity, we write $V^{(L)}(\x)=V(\x)$ and $\L=V(D_x)$ in this proof. 
By direct computation of the Fourier transform, we have 
\[
x_j e^{i\s\L} u = -\s (\pa_{\x_j}V)(D_x) e^{i\s\L} u+ e^{i\s\L}(x_j u),
\]
and this implies 
\[
\jap{x}^{-1}(\pa_{\x_j} V)(D_x) e^{i\s\L}u = -\s^{-1}\bigbra{x_j \jap{x}^{-1} e^{i\s\L}u 
-\jap{x}^{-1}e^{i\s\L}(x_j u)}\in L^2(\re^d). 
\]
By the assumption, we also have 
\[
\sum_{j=1}^d \bigabs{\pa_{\x_j} V(\x)} \geq c|\x|^{\b-1}, \quad |\x|\geq 1, 
\]
with some $c>0$, and these imply $\jap{x}^{-1} e^{i\s\L}u\in H^{\b-1}(\re^d)$, 
and its norm is bounded by $\norm{\jap{x}u}$. This proves \eqref{eq4-3} with $N=1$. 
Similarly we have 
\[
x_j^2 e^{i\s\L} u = \s^2(\pa_{\x_j}V(D_x))^2 e^{i\s\L}u +2\s (\pa_{\x_j}V)(D_x) e^{i\s\L}(x_j u)
+ e^{i\s\L} (x_j^2 u).
\]
This implies $\jap{x}^{-2}((\pa_{\x_j}V)(D_x))^2 e^{i\s\L}u\in L^2(\re^d)$ since we already 
know  $\jap{x}^{-1}(\pa_{\x_j} V)(D_x) e^{i\s\L}u\in L^2(\re^d)$. 
Summing up these estimates in $j$, we learn $\jap{x}^{-2}e^{i\s\L}u\in H^{2(\b-1)}(\re^d)$
and we obtain \eqref{eq4-3} for $N=2$. 
Iterating this procedure, we conclude \eqref{eq4-3} for any $N$. 
\end{proof}

The conclusion of Theorem~5 for $t<0$ now follows from \eqref{eq4-2} and Lemma~11. 
The case $t>0$ is proved similarly. 
\end{proof}

\end{document}